\documentclass[12pt,english]{article}
\usepackage{amsmath, amsthm, latexsym, amssymb,url}
\usepackage{amsfonts}
\usepackage{xypic}
\usepackage{epsfig}

\usepackage{graphicx}% http://ctan.org/pkg/graphicx
\usepackage{yhmath}% http://ctan.org/pkg/yhmath
\usepackage{mathdots}% http://ctan.org/pkg/mathdots
\usepackage{mathtools} 
\usepackage{pifont}

\input xy
\xyoption{all}

\newcommand{\shrinkmargins}[1]{
  \addtolength{\textheight}{#1\topmargin}
  \addtolength{\textheight}{#1\topmargin}
  \addtolength{\textwidth}{#1\oddsidemargin}
  \addtolength{\textwidth}{#1\evensidemargin}
  \addtolength{\topmargin}{-#1\topmargin}
  \addtolength{\oddsidemargin}{-#1\oddsidemargin}
 \addtolength{\evensidemargin}{-#1\evensidemargin}
  }

\shrinkmargins{.7}

\theoremstyle{plain}

\newtheorem{theorem}{Theorem}[section]
\newtheorem{corollary}[theorem]{Corollary}
\newtheorem{lemma}[theorem]{Lemma}
\newtheorem{proposition}[theorem]{Proposition}

\newtheorem*{teo}{Theorem}

\newtheorem{definition}[theorem]{Definition}

\theoremstyle{remark}
\newtheorem{remark}[theorem]{Remark}

\theoremstyle{definition}

\theoremstyle{fact}

\theoremstyle{claim}

\def \Z { \mathbb{Z}}
\def \Q { \mathbb{Q}}

\def \R { \mathbb{R}}

\begin{document}

\thispagestyle{empty}
\setcounter{tocdepth}{7}

\title{On the Arithmetic determination of the trace.}
\author{Guillermo Mantilla-Soler}

\date{}

\maketitle

\begin{abstract}
Let $K$ be a number field, which is tame and non totally real. In this article we give a numerical criterion, depending only on the ramification behavior of ramified primes in $K$,  to decide whether or not the integral trace of $K$ is isometric to the integral trace of another number field $L$. As a byproduct of our proofs here, and in contrast with our previous results for cubic fields of positive discriminant, we show that for cubic fields of negative discriminant  isometry between integral traces is equivalent to equality of discriminants.\\
\end{abstract}

\section{ Introduction}
 
One of the most basic arithmetic invariants of a number field $K$ is its discriminant. Even though the discriminant is not a complete invariant, at least for degree bigger than $2$, the problem of classifying number fields by means of their discriminants is central in algebraic number theory (\cite{ev}, \cite{cohen} and \cite{Sch}). Since the discriminant of a number field $K$ is the determinant associated to an integral bilinear pairing, the trace pairing on the maximal order, it is particularly interesting to ask when for two number fields such pairings are equivalent. In \cite{Manti} we studied a closely related question for cubic fields, and obtained a complete answer  whenever the fields are totally real and of fundamental discriminant. In \cite{Manti1} we found that for tame number fields with ramification at exactly one odd prime, and  at infinity, the equivalence between the trace pairings is determined by the discriminant and the signature of the fields. In this paper, we show that for number fields ramified at infinity, with only tame ramification, the integral trace pairing is totally determined by the discriminant, the signature, and a finite set of positive integers that depend only on the factorization of ramified primes. In particular, for non-totally real number fields we see that in the absence of wild ramification, the integral trace form is determined by the arithmetic of the number field at ramified primes.  

\subsection{The results}
 Let $K$ be a number field with maximal order denoted by $O_{K}$. We write ${\rm t}_{K}:O_{K} \times O_{K} \to \Z$ for the associated integral bilinear pairing induced by the trace form ${\rm Tr}_{K /\Q}: O_{K} \to \Z$. The {\it integral trace form} $q_{K}$ is the integral quadratic form associated to the pairing $ {\rm t}_{K}$. We say that two number fields $K$ and $L$ have isometric integral trace forms whenever there is an equivalence of integral quadratic forms between the forms $q_{K}$  and $q_{L}$.
Two immediate necessary conditions for the existence of such an isometry  are that the two number fields have the same signatures and discriminants. We add a numerical condition, in terms of ramified primes, that guaranties that such an isometry exists. For every prime $p$ and number field $K$ we define  a positive integer $\alpha_{p}^{K}$ which we call the  {\it first ramification factor} (see Definition \ref{RamInvar} for details). %with the property that $\alpha_{p}^{K}\neq 1$ if and only if $p$ ramifies in $K$. 

\begin{teo}[cf. Theorem \ref{isometrytrace}]
Let $K,L$ be tamely ramified number fields of the same degree over $\Q$, and suppose that $K$ is non-totally real. The integral trace forms of $K$ and $L$  are isometric if and only if the following conditions hold:

\begin{itemize}

\item[i)] $\mathrm{disc}(K)=\mathrm{disc}(L)$,

\item[ii)] $K$ and $L$ have the same number of complex embeddings,

\item[iii)] For every finite prime $p \neq 2$ that ramifies in $K$ we have that \[ \left( \frac{\alpha_{p}^{K}}{p} \right)=\left( \frac{\alpha_{p}^{L}}{p} \right). \]

\end{itemize}
\end{teo}

Theorem \ref{isometrytrace} can also be stated in terms of ramification and inertia degrees at ramified primes (see Proposition \ref{usefulcalc}).\\

When dealing with small dimensional number fields the invariants $\alpha_{p}^{K}$ can be easily determined. An interesting consequence of this is the following result on cubic fields.

\begin{teo}[cf. Theorem \ref{isometrycubic}]
Let $K,L$ be cubic number fields, and suppose $\mathrm{disc}(K)<0$. Then, $K$ and $L$ have isometric integral trace forms if and only if they have the same discriminant.
\end{teo}

\begin{remark}
Notice that Theorem \ref{isometrycubic} does not make any reference to ramification type. In particular, it is valid even in the presence of wild ramification.
\end{remark}

In higher dimensional number fields, in contrast with cubic fields ramified at infinity, the discriminant is not enough to determine the isometry class of the integral trace. However, in certain cases the discriminant captures the spinor genus of the integral trace. More explicitly we have:

\begin{teo}[cf. Theorem \ref{GalTotTame}]
Let $K,L$ be two totally ramified tame Galois number fields. Suppose that both $K$ and $L$ have odd degree. Then, the forms $q_{K}$ and $q_{L}$ belong to the same spinor genus if and only if
$\mathrm{disc}(K)=\mathrm{disc}(L).$ 
\end{teo}

For number fields of fundamental discriminant\footnote{Recall that a number $d$ is called fundamental discriminant whenever it is equal to the discriminant of a quadratic field.} our results  lead us to a surprising way to determine when an isometry between the integral traces exists. It turns out that the existence of such an isometry depends on the parity of the number of factors of ramified primes in each field. 
\begin{teo}[cf. Theorem \ref{arithmeticqequiv}]
Let $K, L$ be non totally real number fields of the same signature and same fundamental discriminant, and assume further that $2$ is at worst tamely ramified in both fields. Then, the integral trace forms of $K$ and $L$ are isometric if and only if for every odd prime $p$ that ramifies in $K$ the number of primes in $O_K$ lying over $p$ has the same parity as and the number of primes in $O_{L}$ lying over $p$.
\end{teo}

\section{Proofs for general degrees}
We start with a general fact about a special kind of quadratic forms over the $p$-adic integers. We will denote by $\langle a_1,...,a_n\rangle$ the isometry class of a quadratic form $a_{1}x_{1}^{2}+...+a_{n}x_{n}^{2}$. \\  

\begin{lemma}\label{genusandhasse} Let $p$ be an odd prime, and let $\alpha_{1}, \alpha_{2}, \beta_{1}, \beta_{2}$ be elements in $\Z_{p}^{*}$ such that $\alpha_{1} \beta_{1}=\alpha_{2} \beta_{2} \bmod (\Z_{p}^{*})^2$. Let $0<f \leq n$ be positive integers and consider the $n$-dimensional $\Z_{p}$-quadratic forms given by 
\[ q(\alpha_{i}, \beta_{i}):=  \underbrace{\langle 1,..., 1,\alpha_{i} \rangle}_{f}  \bigoplus \langle p \rangle \otimes \underbrace{\langle 1,...,1,\beta_{i} \rangle}_{n-f}.\] 
The following are equivalent:

\begin{enumerate}

\item The forms $q(\alpha_{1}, \beta_{1})$ and $q(\alpha_{2}, \beta_{2})$ are isometric over $\Z_{p}.$

\item The forms $q(\alpha_{1}, \beta_{1})$ and $q(\alpha_{2}, \beta_{2})$ are isometric over $\Q_{p}.$

\item $(\alpha_{1},p)_{p}=(\alpha_{2}, p)_{p}$, where $(,)_{p}$ denotes the $p$-adic Hilbert symbol.

\end{enumerate} 

\end{lemma}

\begin{proof}

The implication $(1) \Rightarrow (2)$ is obvious.  Since $(\alpha_{1},p)_{p}=(\alpha_{2}, p)_{p}$ implies that $\alpha_{1} =\alpha_{2} \bmod (\Z_{p}^{*})^2$, from which  $\beta_{1}= \beta_{2} \bmod (\Z_{p}^{*})^2$ follows,  then $(3) \Rightarrow (1)$. To show $(2) \Rightarrow (3)$  we may assume that $f < n$. Let $h_{p}(q)$ be the Hasse-Witt invariant of a form $q$. By the linearity of the Hilbert symbol we have that 
\begin{eqnarray*}
h_{p}(q(\alpha_{i}, \beta_{i}))&=& h_{p}(\langle p,p,..., p,p\beta_{i} \rangle)(\alpha_{i}, p^{n-f}\beta_{i})_{p}\\
&=&(p,p)_{p}^{\frac{(n-f-1)(n-f)}{2}}(\beta_{i},p)_{p}^{n-f-1}(\alpha_{i}, p)^{n-f}_{p}(\alpha_{i},\beta_{i})_{p}\\
&=&(p,p)_{p}^{\frac{(n-f-1)(n-f)}{2}}(\alpha_{i}\beta_{i},p)_{p}^{n-f-1}(\alpha_{i},p)_{p}
\end{eqnarray*}

Since $\alpha_{1} \beta_{1}=\alpha_{2} \beta_{2} \bmod (\Z_{p}^{*})^2$ we have that $h_{p}(q(\alpha_{1}, \beta_{1}))=h_{p}(q(\alpha_{2}, \beta_{2}))$ if and only if $(\alpha_{1},p)_{p}=(\alpha_{2}, p)_{p}$. Thus, $(2) \Leftrightarrow (3)$.

\end{proof}
Throughout the paper we will use Conway notation ($p=-1$) for the prime at infinity  (See \cite[Chapter 15, \S4]{conway}). For $p \neq -1$ we denote by $v_{p}$ the usual $p$-adic valuation in $\Q_{p}.$

\paragraph{Ramification Invariants}

Given a number field $L$  and a prime $p$  we denote by  $g_{p}^{L}$ the number of primes in $O_{L}$ lying over $p$. In particular,  $g_{-1}^{L}= r_{L}+s_{L}$ where $r_{L}$ (resp, $s_{L}$) is the number of real (resp, complex)  embeddings of $L$. Furthermore, 
\[e_{p}^{L}:= \sum_{i=1}^{g_{p}^{L}}e_{i}(L) \ \mbox{and} \ f_{p}^{L}:=\sum_{i=1}^{g_{p}^{L}}f_{i}(L),\] where $e_{1}(L),..., e_{g_{p}^{L}}(L)$ are the ramification degrees of the prime $p$ in $L$,  with respective residue degrees $f_{1}(L),..., f_{g_{p}^{L}}(L)$. \\

When the field $L$ is clear from the context we will denote the ramification (resp, residue) degrees only by $e_{i}$ (resp, $f_{i}$) instead of $e_{i}(L)$ (resp, $f_{i}(L)$).

\begin{definition}

For all primes $p$  we define the integer $u_{p}$ as follows 
\[u_{p}=\begin{cases}
-1 & \mbox{ if $p=-1$,}\\
 \ ~ 5 & \mbox{ if $p=2$,}\\
 \min\limits_{u \in \Z} \left \{ u : \mbox{ $0< u <p$  } | \left( \frac{u}{p} \right)=-1 \right \} & \ \mbox{for every other $p$.}
\end{cases}\]
\end{definition}

\begin{remark}
Notice that $u_{p} \in \Z_{p}^{*} \setminus (\Z_{p}^{*})^2$, in particular for $p\neq 2$ we have that $u_{p}$ is a generator of  $\Z_{p}^{*} / (\Z_{p}^{*})^2$.\\
\end{remark}

\begin{definition}\label{RamInvar} 
Let $L$ be a number field of degree $n$ and let $p$ be a prime. The first and second ramification factors of $p$ in $L$ are the integers defined by:
\begin{align*}
\alpha_{p}^{L} &:= \left(\prod_{i=1}^{g_{p}^{L}}  e_{i}^{f_{i}}\right) u_{p}^{(f_{p}^{L}-g_{p}^{L})} \\ \beta_{p}^{L} & :=   \left( \left(-1\right)^{\sum_{i=1}^{g_{p}^{L}}\left(\left\lfloor{\frac{(e_{i} -1)}{2}}\right\rfloor f_{i}\right) } \right) \left(\prod_{i=1}^{g_{p}^{L}}  e_{i}^{(e_{i}-f_{i})}\right)  u_{p}^{(n-f_{p}^{L}-e_{p}^{L}+g_{p}^{L})}. 
\end{align*} Furthermore, we denote by $\mathfrak{a}_{p}^{L}$ the integral quadratic form of discriminant $\alpha_{p}^{L}$ given by \[\mathfrak{a}_{p}^{L} := \underbrace{\langle e_{1},...,e_{1},e_{1}(-1)^{f_{1}-1}, e_{1}(-u_{p})^{f_{1}-1}\rangle}_{f_{1}} \oplus...\oplus \underbrace{\langle e_{g_{p}^{L}},...,e_{g_{p}^{L}},e_{g_{p}^{L}}(-1)^{f_{g_{p}^{L}}-1}, e_{g_{p}^{L}}(-u_{p})^{f_{g_{p}^{L}}-1}\rangle}_{f_{g_{p}^{L}}}.\]

\end{definition} 

One of the useful applications of Lemma \ref{genusandhasse} is that it allows to check easily whenever two number fields have integral trace forms in the same genus. This can be achieved thanks to the following result:

\begin{theorem}\label{GenusPaperThm}\cite[Theorem 0.1]{Manti3}
Let $L$ be a degree $n$ number field.  Let $p$ be a prime which is not wildly ramified in $L$. If we denote by $q_{L} \otimes \Z_{p}$ the quadratic form over $\Z_{p}$ induced by  $q_{L}$, then 
\[q_{L} \otimes \Z_{p} \cong 
\begin{cases}
\mathfrak{a}_{p}^{L}  \bigoplus p \otimes \underbrace{( \mathbb{H} \oplus...\oplus \mathbb{H})}_{\frac{n-f_{p}^{L}}{2}} & \mbox{ if $p=2$,}\\
\mathfrak{a}_{p}^{L}  \bigoplus p \otimes \underbrace{\langle 1,...,1,\beta_{p}^{L}\nu_{p}^{L} \rangle}_{n-f_{p}^{L}}  &\mbox{ if $p \neq 2$}. \\
\end{cases}\]
Furthermore, if $p \neq 2$ we have that $\mathfrak{a}_{p}^{L} \cong \underbrace{\langle1,....,1,\alpha_{p}^{L} \rangle}_{f_{p}^{L}}.$ Here $\mathbb{H}$ denotes the Hyperbolic plane over $\Z_{2}$ and for an odd prime $\nu_{p}^{L}$ is the unique element in $\Z_{p}^{*}/(\Z_{p}^{*})^{2}$ such that $\displaystyle  p^{n-f_{p}^{L}}\alpha_{p}^{L} \beta_{p}^{L} \nu_{p}^{L}={\rm disc}(L).$

\end{theorem}

\begin{remark}\label{Serre's}
From the above it follows that whenever $L$ is a tamely ramified number field of degree $n$ its discriminant is given by 
\[ {\rm disc}(L) = \prod_{p}p^{(n -f_{p}^{L})}.\]
For details see \cite[Corollary 1.14]{Manti3} or \cite[Chapter III, Proposition 13]{Serre}. For $p=-1$ the above says that the sign of the discriminant is $(-1)^{s_{L}}$. \end{remark} 

The following result, due to B.Erez, J. Morales and R.Perlis, shows that the integral trace over $\Z_{2}$ is completely determined by the discriminant and the degree of the field whenever $2$ is at worst tamely ramified.

\begin{proposition}\label{2genustame}
Let $K,L$ be number fields of the same degree and discriminant. Suppose that $2$ is not wildly ramified in either of them. Then, \[q_{K} \otimes \Z_{2} \cong q_{L} \otimes \Z_{2}.\]
\end{proposition}

\begin{proof}
By the ramification hypothesis, and thanks to \cite[8.5]{conneryui}, we have that \[q_{K} \otimes \Q_{2} \cong q_{L} \otimes \Q_{2}.\] Since both fields are at worst tame at $2$, and have the same degree and discriminant, we have, thanks to Remark \ref{Serre's}, that $f_{2}^{K}=f_{2}^{L}$. It follows from Theorem \ref{GenusPaperThm} and Witt's cancellation theorem that the forms $\mathfrak{a}_{2}^{K}$ and $\mathfrak{a}_{2}^{L}$ are equivalent over $\Q_{2}$. Since  $\mathfrak{a}_{2}^{K}$ and $\mathfrak{a}_{2}^{L}$ are odd forms we have from \cite[93:16]{Om} that they are equivalent over $\Z_{2}$ and hence the result follows.

\end{proof}

The following is a generalization of \cite[Theorem 1.2]{Manti1}.

\begin{theorem}
Let $K,L$ be two non-totally real tamely ramified number fields of the same signature and discriminant. Suppose that at most one odd prime ramifies in $K$. Then, the integral quadratic forms $q_{K}$ and $q_{L}$ are isometric.
\end{theorem}

\begin{proof}
Since the discriminant is a complete invariant for quadratic fields, we may assume that the degree of $K$ and $L$ is greater than $2$. By the hypothesis on the signatures we have that the forms $q_{K}$ and $q_{L}$ are isometric over $\R$. They are also isometric over $\Z_{2}$ thanks to Proposition \ref{2genustame}, and by the classification of unimodular forms over $\Z_{\ell}$ they are also isometric over $\Z_{\ell}$ for every odd prime $\ell$ not dividing the common discriminant. Finally, by the product formula on local Hasse invariants we have that $q_{K}$ and $q_{L}$ are isometric over $\Q_{p}$ for every $p$. From Theorem \ref{GenusPaperThm} and Lemma \ref{genusandhasse} we have that $q_{K}$ and $q_{L}$ are in the same genus, so by \cite[Theorem 2.12]{Manti2} they are in the same spinor genus. Since the fields are non-totally real, the forms $q_{K}$ and $q_{L}$ are two regular indefinite forms of dimension at least $3$ that have the same spinor genus.  By Eichler's \cite{eichler}  we have that they are integrally equivalent.
\end{proof}

\begin{proposition}\label{conditionssameg}

Let $K,L$ be tamely ramified number fields of degree $n \ge 3$. The integral quadratic forms $q_{K}$ and $q_{L}$ belong to the same spinor genus if and only if the following conditions hold:

\begin{itemize}

\item[i)] $\mathrm{disc}(K)=\mathrm{disc}(L)$,

\item[ii)] $s_{K}=s_{L}$,

\item[iii)] For every finite prime $p \neq 2$ that divides the common discriminant of $K$ and $L$ we have that \[ \left( \frac{\alpha_{p}^{K}}{p} \right)=\left( \frac{\alpha_{p}^{L}}{p} \right). \]

\end{itemize}

\end{proposition}

\begin{proof}
By Proposition \ref{2genustame} we have that $q_{K} \otimes \Z_{2} \cong q_{L} \otimes \Z_{2}$, and since both fields have the same signature we have that their local traces coincide at $p=-1$. Since $K$ and $L$ have the same discriminant, and both fields are tame, we have thanks to Lemma \ref{genusandhasse} and Theorem \ref{GenusPaperThm} that $q_{K} \otimes \Z_{p} \cong q_{L} \otimes \Z_{p}$ for all other values of $p$. The results follows from \cite[Theorem 2.12]{Manti2}.

\end{proof}

\begin{corollary}\label{GalTame}
Let $K,L$ be two tame Galois number fields of the same odd degree. For a given prime $p$ let $e_{p}$ (resp, $\tilde{e}_{p}$) be the ramification degree of $K$ (resp, $L$) at $p$. The forms $q_{K}$ and $q_{L}$ belong to the same spinor genus if and only if

\[\mathrm{disc}(K)=\mathrm{disc}(L) \ {\rm and} \   \left( \frac{e_{p}}{p} \right)=\left( \frac{\tilde{e}_{p}}{p} \right)\] for all odd primes $p$ that ramify in $K$.
\end{corollary}

\begin{proof}
Since $K$ is an odd Galois extension it follows from the definition of the first ramification invariant that $\alpha_{p}^{K} \equiv e_{p} \bmod (\Z_{p}^{*})^{2}$, and similarly for $L$. Since odd Galois extensions are totally real the result follows from Proposition \ref{conditionssameg}.
\end{proof}

\begin{theorem}\label{GalTotTame}
Let $K,L$ be two totally ramified tame Galois number fields. Suppose that both $K$ and $L$ have odd degree. Then, the forms $q_{K}$ and $q_{L}$ belong to the same spinor genus if and only if
$\mathrm{disc}(K)=\mathrm{disc}(L).$ 
\end{theorem}

\begin{proof}
Suppose that $\mathrm{disc}(K)=\mathrm{disc}(L)$ and let $p$ be a prime ramifying in either field. Since both fields are totally ramified we have that $[K:\Q]=e_{p}$ and $[L:\Q]=\tilde{e}_{p}$. On the other hand since both fields are tame we have that $v_{p}(\mathrm{disc}(K))=e_{p}-1$ and $v_{p}(\mathrm{disc}(L))=\tilde{e}_{p}-1$. In particular,  $[K:\Q]= [L:\Q]$ and $e_{p}=\tilde{e}_{p}$. The result follows from Corollary \ref{GalTame}.
\end{proof}

\begin{remark}
One particular case in which Theorem \ref{GalTotTame} applies is when both fields $K$ and $L$ are tame $\Z/\ell\Z$-extensions of $\Q$ for some prime $\ell$. In such a case there is a stronger result of Conner and Perlis which says that $q_{K}$ and $q_{L}$ are isometric whenever they have the same discriminant. Moreover, in their result the fields can have wild ramification (see \cite[Chapter IV]{conner}).
\end{remark}

\begin{theorem}\label{isometrytrace}
Let $K,L$ be tamely ramified number fields of the same degree, and suppose that $s_{K}>0$. The integral quadratic forms $q_{K}$ and $q_{L}$ are isometric if and only if the following conditions hold:

\begin{itemize}

\item[i)] $\mathrm{disc}(K)=\mathrm{disc}(L)$,

\item[ii)] $s_{K}=s_{L}$,

\item[iii)] For every finite prime $p \neq 2$ that divides the common discriminant of $K$ and $L$ we have that \[ \left( \frac{\alpha_{p}^{K}}{p} \right)=\left( \frac{\alpha_{p}^{L}}{p} \right). \]

\end{itemize}
\end{theorem}

\begin{proof}
Since the spinor genus of  a non-positive definite form of degree at least $3$ contains only one isometry class the result follows from Proposition \ref{conditionssameg}.
\end{proof}

\begin{proposition}\label{usefulcalc}

Let $K,L$ be tamely ramified number fields of the same degree, and suppose that $s_{K}>0$. The integral quadratic forms $q_{K}$ and $q_{L}$ are isometric if and only if the following conditions hold:

\begin{itemize}

\item[(a)]  $f_{p}^{K}=f_{p}^{L}$ for every prime $p$ that ramifies in either $K$ or $L$.

\item[(b)] For every finite prime $p \neq 2$ that ramifies in either $K$ or $L$  \[g_{p}^{K}-h_{p}^{K} \equiv  g_{p}^{L}-h_{p}^{L} \pmod{2},\] where 
\[  h_{p}^{L} =\# \left \{i : \mbox{ $f_{i}(L)$ is odd and $\left( \frac{e_{i}(L)}{p} \right)=-1$} \right \}.\]

\end{itemize}

\end{proposition}

\begin{proof}
The result is implied by Theorem \ref{isometrytrace} and the following three observations about tame number fields.
Let $F$ be a degree $n$ tamely ramified number field. Then,  \begin{itemize}
\item  $(-1)^{f^{F}_{p}}\left(  \frac{\alpha^{F}_{p}}{p} \right)= (-1)^{g_{p}^{F}-h_{p}^{F}}$. This follows from the multiplicative properties of the Jacobi symbol and since $\left( \frac{u_{p}}{p} \right)=-1$. \\

\item $ n -f_{-1}^{F}=s_{F}$. Every place above $p=-1$ has inertia degree $1$ hence $f_{-1}^{F}$ is the number of places above infinity i.e., $r_{F}+s_{F}$.\\

\item $\displaystyle {\rm disc}(F) = \prod_{p}p^{(n -f_{p}^{F})}.$ See Remark \ref{Serre's}.

\end{itemize}

\end{proof}

For fields of fundamental discriminant the above characterization of the isometry between trace forms can be greatly simplified.

\begin{theorem}\label{arithmeticqequiv}
Let $K, L$ be tame non-totally real number fields of the same signature and same fundamental discriminant. Then, the integral quadratic forms $q_{K}$ and $q_{L}$ are isometric if and only if for every odd prime that ramifies in $K$ we have that \[g_{p}^{K} \equiv  g_{p}^{L}\pmod{2}.\] 
\end{theorem}

\begin{proof}
Let $p$ be an odd prime ramified in $K$, and let $d:={\rm disc}(K)$. By the hypothesis on $K$ we have that $v_{p}(d)=1$ and since $v_{p}(d)=[K:\Q] -f_{p}^{K}$ it follows that \[f_{1}(K)(e_{1}(K)-1)+...+f_{g_{p}^{K}}(K)(e_{g_{p}^{K}}(K)-1)=1.\] Therefore there exists a unique $1 \leq i \leq g_{p}^{K}$ such that $e_{i} \neq 1$. Moreover for such an $i$ we have that $f_{i}=1$ and $e_{i}=2$, thus \[  h_{p}^{K} =\# \left \{i : \mbox{ $f_{i}(K)$ is odd and $\left( \frac{e_{i}(K)}{p} \right)=-1$} \right \}= \frac{1-\left( \frac{2}{p} \right)}{2}.\] Similarly we have that  $\displaystyle h_{p}^{L}=\frac{1-\left( \frac{2}{p} \right)}{2}$ so in particular $h_{p}^{K}=h_{p}^{L}$. Since $K$ and $L$ have the same signatures we have that $[K:\Q]=[L:\Q]$, and since they have the same discriminant and both are tame we have that $f_{p}^{K}=f_{p}^{L}$ for every ramified prime $p$. Since $h_{p}^{K}=h_{p}^{L}$ for every odd ramified prime the result follows from Proposition \ref{usefulcalc}.
\end{proof}

\begin{remark}
Notice that a number field that has fundamental discriminant can only have wild ramification at $p=2$. In particular, for square free discriminants we can remove the tameness condition on the above theorem.\end{remark}

\section{Cubic fields}

All the results we have proved on the isometry of the integral trace have the assumption that the number fields are tamely ramified. In the case of cubic fields we can improve this by giving sufficient and enough conditions to decide when two cubic fields have integral traces in the same spinor genus. As it turns out, the equality between the discriminants gives such necessary and sufficient  conditions.
\begin{lemma}\label{localcubicos}
Let $F_{1},F_{2}$ be two totally ramified cubic extensions of $\Q_{3}$. Then, the integral trace forms of $F_{1}$ and $F_{2}$ are equivalent if and only if \[ {\rm disc}(F_{1})={\rm disc}(F_{2}) \bmod (\Z_{3}^{*})^2.\]
\end{lemma}

\begin{proof} Let $O_{i}$ be the ring of integers of $F_{i}$, and let $T_{F_{i}}$ be the integral trace form of $F_{i}$. Since $F_{i}/\Q_{3}$ is totally ramified we have that $O_{i}=\Z_{3}[\alpha_{i}]$ where $\alpha_{i}$ satisfies a polynomial of the form $x^{3}+3a_{i}x+b_{i} \in \Z_{3}[x].$ From this we see that ${\rm tr}_{F_{i}/\Q_{3}}(O_{i}) \subset 3\Z_{3}$ thus \[T_{F_{i}} \cong \langle 3 \rangle \otimes U_{i}\] where $U_{i}$ is a ternary quadratic form over $\Z_{3}$. More explicitly, if one considers the $\Z_{3}$-basis of $O_{i}$ given by $\{1,\alpha_{i}, \alpha_{i}^2+2a_{i} \}$ then the Gram matrix of the trace form in that basis is given by \[\left[ \begin{array}{ccc}3 & 0 & 0 \\0 & -6a_i & -3b_i \\0 & -3b_i & 6a_{i}^2\end{array}\right].\] Suppose that $F_{1}$ and $F_{2}$ have the same discriminant, hence the same is true about $U_{1}$ and $U_{2}$. Hence if $d:={\rm disc}(U_{i})$ we see that $d =\frac{{\rm disc}(F_{i})}{27}=-(b_{i}^2+4a_{i}^{3}).$ Since $v_{3}({\rm disc}(F_{i})) \in \{3,4,5\}$, \cite[Chapter III, Proposition 13]{Serre}, we have that $v_{3}(d) \in \{0,1,2\}$. 
\begin{itemize}

\item If $v_{3}(d)=0$ then $U_{1} \cong U_{2} \cong \langle 1,1,d\rangle$ or equivalently \[T_{F_{1}} \cong T_{F_{2}} \cong \langle 3,3,3d\rangle.\]

\item If $v_{3}(d)=1$ then $F_{1}$ and $F_{2}$ are among the four cubic extensions of $\Q_{3}$ with discriminant with 3-valuation equal to $4$. A set of polynomials defining these extensions is: 

\[ x^{3}-3x+19, \quad x^{3}-3x+1, \quad  x^{3}-3x+10, \quad x^{3}-3x+5.\]

An explicit calculation shows that the integral trace form of all these extensions is of the form $\langle 3,6,9\delta\rangle$ for some  $\delta \in \Z_{3}^{*}$. Therefore any two of these extensions have the same integral trace form if and only if they have the same discriminant. In particular, \[T_{F_{1}} \cong T_{F_{2}} \cong \langle 3,6,3d/2\rangle.\]

\item If $v_{3}(d)=2$ then $F_{1}$ and $F_{2}$ are among the three cubic extensions of $\Q_{3}$ with discriminant having 3-valuation equal to $5$. A set of polynomials defining these extensions are: 

\[ x^{3}+3, \quad x^{3}+21, \quad x^{3}+12.\]

An explicit calculation shows that the integral trace of all these extensions is given by $\langle 3,9,-9\rangle$. In particular  \[T_{F_{1}} \cong T_{F_{2}} \cong \langle 3,9,-9\rangle.\]

\end{itemize}
\end{proof}

\begin{theorem}\label{cubicos}
Let $K,L$ be cubic number fields. Then $q_{K}$ and $q_{L}$ belong to the same spinor genus if and only if $\mathrm{disc}(K)=\mathrm{disc}(L)$.
\end{theorem}

\begin{proof} Since the discriminant is an invariant of a genus, we have that  $\mathrm{disc}(K)=\mathrm{disc}(L)$ if $q_{K}$ and $q_{L}$ belong to the same spinor genus. Suppose now that $K$ and $L$ have the same discriminant. Thanks to \cite[Theorem 2.12]{Manti2} it is enough to show that for every prime $p$, the integral forms $q_{K} \otimes \Z_{p}$ and $q_{L} \otimes \Z_{p}$ are equivalent.

\begin{itemize}

\item[(a)] If either $p=-1$ or $p$ does not divide the common discriminant $d$  we have that \[q_{K} \otimes \Z_{p} \cong \langle 1,1, d \rangle \cong q_{L} \otimes \Z_{p}.\]

\item[(b)] Let $p \nmid 6$ be a finite prime dividing the common discriminant. Since $p$ is tamely ramified in both fields we see by looking at the valuation of the discriminant at $p$ that $f_{p}^{K}=f_{p}^{L}$. This equality implies that $p$ has the same number of prime factors and  same ramification and residue degrees in both $K$ and $L$. Thus $h_{p}^{K}=h_{p}^{L}$, and $g_{p}^{K}=g_{p}^{L}$. Hence the result follows by Proposition \ref{usefulcalc}. \\

\item[(c)] Let $p=2$. If $2$ has at worst tame ramification in both fields we have, thanks to Proposition \ref{2genustame}, that $q_{K} \otimes \Z_{2} \cong q_{L} \otimes \Z_{2}$.  If $2$ has wild ramification in both $K$ and $L$ then \[q_{K} \otimes \Z_{2} \cong  \langle 1 \rangle \oplus {\rm T}_{F}  \cong q_{L} \otimes \Z_{2},\] where $F$ is the unique quadratic extension of $\Q_{2}$ with discriminant  equal to the common discriminant of $K$ and $L$. We finish by showing that the above are the only two possibilities i.e.,

{\bf Claim:} {\it The ramification type of $2$ is the same in $K$ and $L$}. 

{\it Proof.} If $2$ had wild ramification in $K$ and tame in $L$ then its prime factorization would be of the form $\mathcal{B}_{1}\mathcal{B}_{2}^2$ in $O_{K}$ and of the form $\mathcal{B}^{3}$ in $O_{L}$. Therefore  \[q_{K} \otimes \Z_{2} \cong  \langle 1 \rangle \oplus {\rm T}_{F},\] where ${\rm T}_{F}$ is the integral trace form of a quadratic ramified extension $F/\Q_{2}$. On the other hand, thanks to Theorem \ref{GenusPaperThm}, we have that \[ q_{L} \otimes \Z_{2} \cong  \langle 3\rangle \bigoplus 2 \otimes \mathbb{H}.\]  Since det$\displaystyle\left(\langle 3\rangle \bigoplus 2 \otimes \mathbb{H}\right)=-12 \bmod (\Z_{2}^{*})^2$ we conclude that disc$(F)=-12 \bmod (\Z_{2}^{*})^2$ but there is no quadratic extension $F/\Q_{2}$ with such a discriminant.

\item[(d)] Let $p=3$. We may assume that $p$ has wild ramification, otherwise we could argue as in (a) or (b). Since $3$ has factorization of type $\mathcal{B}^{3}$ in both fields, then $q_{K} \otimes \Z_{3} \cong  {\rm T}_{F_1}$ and $q_{K} \otimes \Z_{3} \cong  {\rm T}_{F_{2}}$ where the $F_{i}$'s are totally ramified cubic extensions of $\Q_{3}$, and ${\rm T}_{F_{i}}$ is the integral trace form of $F_{i}$. By Lemma \ref{localcubicos} we have that ${\rm T}_{F_1} \cong {\rm T}_{F_2}$, thus \[q_{K} \otimes \Z_{3} \cong q_{L} \otimes \Z_{3}.\] \end{itemize} \end{proof}

As a corollary we obtain that in the case of non-totally real cubic fields, the trace form is completely determined by its discriminant.

\begin{theorem}\label{isometrycubic}
Let $K,L$ be cubic number fields, and suppose $\mathrm{disc}(K)<0$. Then $q_{K}$ and $q_{L}$ are isometric if and only if $\mathrm{disc}(K)=\mathrm{disc}(L)$.
\end{theorem}

\begin{proof}
Since for cubic fields, having a negative discriminant is equivalent to having an indefinite trace form, the result is an immediate consequence of Theorem \ref{cubicos} and Eichler's result on indefinite forms \cite{eichler}.
\end{proof}

\section*{Acknowledgements}

In the first place I would like to thank the referee for the careful reading of the paper, and all of her/his helpfull suggestions. I also thank  Lisa (Powers) Larsson for her valuable comments on a previous version of this paper.

\noindent
Guillermo Mantilla-Soler\\
Departamento de Matem\'aticas,\\
Universidad de los Andes, \\
Carrera 1 N. 18A - 10, Bogot\'a, \\
Colombia.\\
g.mantilla691@uniandes.edu.co

\end{document}